\documentclass[10pt]{article}
\usepackage[toc,page]{appendix}
\usepackage{amsthm,amssymb}
\usepackage{amsmath}
\usepackage{commath}
\usepackage{fnpct}
\usepackage{graphicx,tikz,pgf}
\usepackage{color}
\usepackage{xcolor}
\usepackage{caption}
\usepackage{xpatch}
\usepackage[T2A,T1]{fontenc}
\usepackage[utf8]{inputenc}
\usepackage[russian,english]{babel}
\usepackage{subcaption}
\usepackage{fancyhdr,cancel,enumerate,array,booktabs,setspace}
\usetikzlibrary{patterns,arrows,decorations.pathreplacing}
\usepackage{mathrsfs}
\usetikzlibrary{arrows}
\usepackage{epic}
\usepackage{eepic}
\usepackage{shadow}
\definecolor{xdxdff}{rgb}{0,0,0}
\definecolor{qqqqff}{rgb}{0,0,0}
\definecolor{uuuuuu}{rgb}{0,0,0}
\definecolor{ududff}{rgb}{0,0,0}
\date{ }
\theoremstyle{definition}
\xpatchcmd{\proof}{\itshape}{\normalfont\proofnamefont}{}{}
\newcommand{\proofnamefont}{}
\renewcommand{\proofnamefont}{\bfseries}
\newcommand{\RNum}[1]{\uppercase\expandafter{\romannumeral #1\relax}}

\newtheorem{definition}{Definition}
\newtheorem{remark}[definition]{Remark}
\newtheorem{example}[definition]{Example}

\newtheorem{theorem}[definition]{Theorem}
\newtheorem{lemma}[definition]{Lemma}

\newcommand{\emat}[1]{\mathcal{E}(#1)}
\newcommand{\se}[2][]{\mathcal{E}_{#1}(\mathcal{S}(#2))}
\newcommand{\sm}[1]{\mathcal{S}(#1)}
\newcommand{\nop}[1]{s(#1)}

\title{\bf An improved lower bound for \\the Seidel energy of tree graphs}
\author{\bf \small{M. Einollahzadeh$^{{\rm a}}$ and M.A. Nematollahi$^{{\rm b}}$}\footnote{
		Corresponding author.\newline
		E-mail addresses:
		m$\_$einollahzadeh@cc.iut.ac.ir (M. Einollahzadeh),  ma.nematollahi@fasau.ac.ir (M.A. Nematollahi)}
	\\[2mm]
	${\rm ^{a}}$\small Department of Mathematical Sciences, Isfahan University of Technology, Isfahan, Iran
	\\
	${\rm ^{b}}$\small Department of Computer Sciences, Fasa University, Fasa, Iran
	}

\begin{document}

\maketitle
\begin{abstract}
	Let $G$ be a graph with the vertex set $ \lbrace v_1,\ldots,v_n \rbrace$.
	The Seidel matrix of $G$ is an $n\times n$ matrix whose diagonal entries are zero, $ij$-th entry is $-1$ if
	$ v_{i} $ and $ v_{j} $ are adjacent and otherwise is
	$ 1 $. The Seidel energy of $G$, denoted by $ \se{G} $, is defined to be the sum of absolute values of all eigenvalues of the Seidel matrix of $G$. In \cite{aekn}, the authors proved that the Seidel energy of any graph of order $n$ is at least $2n-2$. In this study, we improve the aforementioned lower bound for tree graphs.
\end{abstract}
\vskip 3mm

\vskip 3mm

\section{Introduction and Terminology}
Throughout this paper all graphs we consider are simple and finite. For a graph $ G $, we denote the set of vertices and edges of $ G $ by $ V(G) $ and $ E(G) $, respectively. The \textit{order} of a graph is its number of vertices and the \textit{size} of a graph is its number of edges. The complement of $ G $ is denoted by $ \overline{G} $. Also, the $ n$-vertex complete graph, path graph and star graph are denoted by $ K_{n} $, $ P_{n} $ and $ S_{n} $, respectively. The \textit{distance} between two vertices in a graph is the number of edges in a shortest path connecting them. Moreover, $ \delta(G) $ and $ \Delta(G) $ represent the minimum degree and the maximum degree of $ G $, respectively.

For every Hermitian matrix $A$ the \textit{energy} of $A$, $\emat{A}$,  is defined to be sum of the absolute values of the eigenvalues of $A$. The well-known concept of energy of a graph $ G $, denoted by $ \mathcal{E}(G) $, is the energy of its adjacency matrix.

Let $ G $ be a graph and $ V(G)=\lbrace v_{1}, \ldots , v_{n} \rbrace $. The \textit{Seidel matrix} of $ G $, denoted by $ \sm{G} $, is an $ n \times n $ matrix whose diagonal entries are zero, $ij$-th entry is $-1$ if $ v_{i} $ and $ v_{j} $ are adjacent and otherwise is $ 1 $ (It is noteworthy that at first, van Lint and Seidel introduced the concept of Seidel matrix for the study of equiangular lines in \cite{lint-seidel}). The \textit{Seidel energy} of $ G $ is defined to be $\se{G}$. Moreover, the \textit{Seidel switching} of $ G $ is defined as follows: Partition $ V(G) $ into two subsets $ V_{1} $ and $ V_{2} $, delete the edges between $ V_{1} $ and $ V_{2} $ and join all vertices $ v_{1} \in V_{1} $ and $ v_{2} \in V_{2} $ which are not adjacent. Therefore, if we call the new graph by $ G^{\prime} $, then we have $\sm{G^{\prime}}= D \sm{G} D $, where $ D $ is a diagonal matrix with entries 1 (resp. $ -1 $) corresponding to the vertices of $ V_{1} $ (resp. $ V_{2} $) (\cite{Ham}). Hence, $ \sm{G} $ and $ \sm{G^{\prime}} $ are similar and they have the same Seidel energy. Note that if one of the $ V_{1} $ or $ V_{2} $ is empty, then $ G $ remains unchanged and also, for every $ v \in V(G) $, using a Seidel switching on $ G $, one can convert $ v $ to an isolated vertex. Two graphs $ G_{1} $ and $ G_{2} $ are called \textit{SC-equivalent} if $ G_{2} $ is obtained from $G_1$ or $ \overline{G_{1}} $ by a Seidel switching and is denoted by $ G_{1} \cong G_{2} $. Note that in either cases, $ \sm{G_{2}} $ is similar to $ \sm{G_{1}} $ or $ -\sm{G_{1}} $, hence $ \se{G_{1}}=\se{G_{2}} $. 

If $ X $ and $ Y $ are two disjoint subsets of $ V(G) $, the set of edges of $ G $ with one endpoint in $ X $ and another in $ Y $ is denoted by $ E(X, Y) $. An ordered pair $ (X,Y) $ of disjoint subsets of $ V(G) $ with $ \vert X \vert = \vert Y \vert =2 $, is called an \textit{odd pair} if $ \vert E(X,Y) \vert $ is an odd number (which is either 1 or 3). One can easily see that applying a Seidel switching on an arbitrary graph $ G $ does not change its odd pair(s). A subset $ \lbrace u, v \rbrace \subseteq V(G) $ is called an \textit{odd set} if $ \lbrace u, v \rbrace $ is the first component of an odd pair of $ G $. We denote the number of odd pairs in $ G $ by $ \nop{G} $.

From a graph $ G $, we construct a graph denoted by $ \Lambda(G) $, as follows: $ V(\Lambda(G)) = V(G)$ and $ E(\Lambda(G))  $ consists of all the edges $ e = uv $ such that $ \lbrace u, v \rbrace $ is an odd set of $ G $. By $ \lambda(v) $ we denote the degree of vertex $ v \in V(G) $ in the graph $ \Lambda(G) $ 
\begin{example}
	The following identities hold which can be easily checked:
	\begin{enumerate}
		\item[1.] 
		$ \Lambda(C_{4}) = \overline{K_{4}} $. 
		\item[2.]
		$ \Lambda(P_{4}) = C_{4} $.
		\item[3.]
		$ \Lambda(C_{5}) = K_{5} $.
		\item[4.]
		$ \Lambda(K_{n}) = \overline{K_{n}} $.
	\end{enumerate}
\end{example}

Haemers in \cite{Ham} introduced the concept of Seidel energy of a graph and he conjectured that for every graph $ G $ of order $ n $, 
\begin{equation}\label{lb}
\se{G} \geq \se{K_{n}} = 2n-2.
\end{equation}
Several attempts were made to prove the conjecture (see for example \cite{gho} and \cite{obu}) and the conjecture was proved in \cite{aekn}. Here, we strengthen inequality \eqref{lb} for tree graphs.

\section{Main Theorem}

In this section, we propose some lemmas to prove the main result of the paper. First, we express the following lemma which deals with the non-adjacent vertices in $ \Lambda(T) $, where $ T $ is a tree of order at least 4.
\begin{lemma}\label{nl-1}
	Let $ T $ be a tree of order $ n \geq 4 $. Assume that $ u $ and $ v $ are two vertices of $ T $ which are non-adjacent in $ \Lambda(T) $. Then, exactly one of the following cases occurs:
\vskip 2mm
\noindent \textbf{Case 1.} $ u $ and $ v $ are two leaves, both connected to another vertex of $ T $. 
\vskip 2mm
\noindent \textbf{Case 2.} $ T $ is the tree graph depicted in Figure \ref{Case2}.

\begin{figure}[h]

\centering
\includegraphics[]{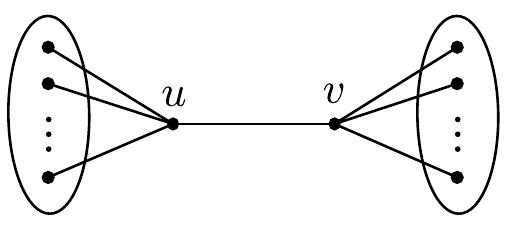}
\caption{Case 2}
\label{Case2}
\end{figure}
\vskip 2mm
\noindent \textbf{Case 3.} $ T $ is the tree graph depicted in Figure \ref{Case3} .

\begin{figure}[h]

\centering
\includegraphics[]{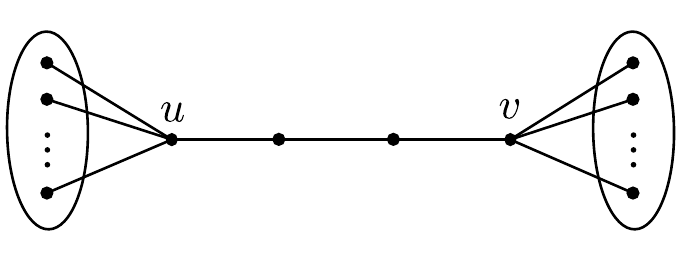}
\caption{Case 3}
\label{Case3}
\end{figure}
\end{lemma}
\begin{proof}
	Let $ d $ be the distance between $ u $ and $ v $ in $ T $. Then there exists a path $ P $ as an induced subgraph of $ T $ which connects $ u $ and $ v $.
	
	If $ d \geq 4 $, then suppose that $ w $ is the vertex adjacent to $ u $ in $ P $ and $ x $ is another vertex in $ P $ which is neither connected to $ u $ nor $ v $ (since $ d \geq 4 $, $ x $ exists). Therefore, $ (\lbrace u, v \rbrace, \lbrace w, x \rbrace) $ is an odd pair, a contradiction.
	
	If $ d = 3 $, then every vertex of $ V(T)\setminus V(P) $ is adjacent to exactly one of the two vertices $ u $ and $ v $; because if a vertex $ w $ is connected to both $ u $ and $ v $, we have cycle in $ T $, a contradiction. On the other hand, if there exist some vertices in $ T $ which are neither connected to $ u $ nor $ v $, without loss of generality, there is a vertex $ x $ which has distance 2 from $ u $ ($ x $ and $ v $ are non-adjacent). Let $ w $ be the vertex joining $ u $ and $ x $ together. Then $ (\{ u, v \}, \{w, x\}) $ is an odd pair of $ T $, a contradiction. Hence, $ T $ is the tree depicted in Figure \ref{Case3}.
	
	Next, assume that $ d = 2 $. Hence, there is a vertex in $ T $ which connects $ u $ to $ v $. Name it $ w $. Then every vertex of $ T $ is either connected to both $ u $ and $ v $ or is neither connected to $ u $ nor $ v $; otherwise, without loss of generality, there exists a vertex $ x $ which is connected to $ u $ and, $ x $ and $ v $ are non-adjacent, which implies that $ (\{u, v\}, \{w, x\}) $ is an odd pair, a contradiction. Moreover, since $ T $ is a tree, $ u $ and $ v $ have exactly one common neighbor (the vertex $ w $). Therefore the first case occurs.
	
	Finally, suppose that $ d = 1 $. So, $ u $ and $ v $ are adjacent. Now, similar to the discussion in case $ d = 3 $, it can be verified that the assumptions of lemma imply that $ T $ is the tree depicted in Figure \ref{Case2}. 

\end{proof}
\begin{definition}
	Let $ T $ be a tree graph. Then, by $ D(T) $ or simply $ D $, we mean the maximum number of leaves which are connected to a vertex of $ T $. Obviously, $$ 1 \leq D(T) \leq \Delta(T),$$
	which are achieved (for example) by the path graphs and the star graphs, respectively.
\end{definition}
Now, we define two types of trees which will be used in the sequel:
\vskip 2mm
%\noindent Type 1:
%Stars which are denoted by $ S_{n} $.
\noindent \textbf{Type 1:} The family of trees of order $ a + b + 2 $ depicted in Figure \ref{Type1}, where $ a \geq b \geq 1 $ are integers.

\begin{figure}[h]

\centering
\includegraphics[]{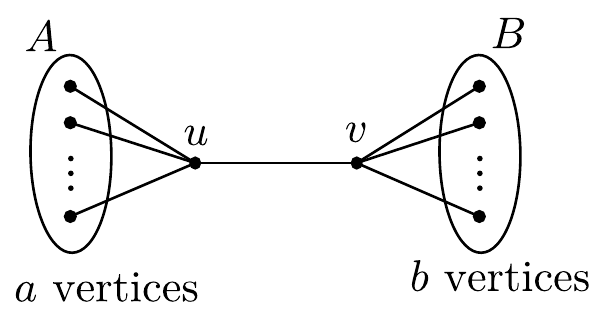}
\caption{Type 1}

\label{Type1}
\end{figure}
\vskip 2mm
\noindent \textbf{Type 2:} The family of trees of order $ a + b + 4 $ depicted in Figure \ref{Type2}, where $ a \geq b $ are \\non-negative integers and $ (a, b) \neq (0, 0) $. (Note that here, the case $ (a, b) = (0, 0) $ yields a Type 1 tree with parameters $ a = b = 1 $.)

\begin{figure}[h]

\centering
\includegraphics[]{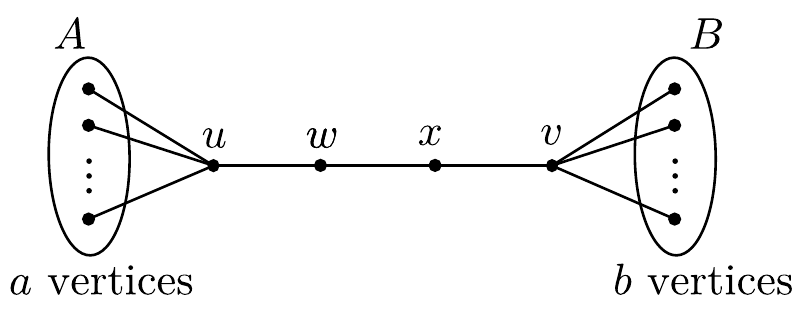}
\caption{Type 2}

\label{Type2}
\end{figure}

\begin{lemma}\label{ntsize}
	Let $ T $ be a tree of order $ n \geq 4 $ which is not $ S_{n} $, $ P_{4} $, $ P_{5} $ and $ P_{6} $. Then we have
	$$ \vert E(\Lambda(T))\vert \geq \frac{1}{2}n(n - D(T)). $$
\end{lemma}
\begin{proof}
	For simplicity, we denote $ D(T) $ by $ D $. Let $ u $ and $ v $ be arbitrary vertices of $ T $. By Lemma \ref{nl-1}, if $ T $ is not any of the star graph or graph of Type 1 or graph of Type 2, then $ uv \notin E(\Lambda(T)) $ implies that $ u $ and $ v $ are two leaves, both connected to another vertex of $ T $. Hence, if $ w $ is not a leaf, then $ \lambda(w) = n - 1$. On the other hand, if $ w $ is a leaf, assume that $ w $ is connected to a vertex $ v $. Then, $ \lambda(w) = n - 1 - (k - 1) = n - k $, where $ k $ is the number of leaves adjacent to $ v $. Therefore, in either case, $ \lambda(w) \geq n - D $, which implies that $$ \vert E(\Lambda(T))\vert \geq \frac{1}{2}n(n - D). $$
	
	Next, suppose that $ T $ is of Type 1. Therefore $ D(T) = a $. Since $ T \neq P_{4} $, we have $ a > 1 $. In this case, $D=a$ and 
	
	\begin{equation*}
	xy \notin E(\Lambda(T)) \Leftrightarrow \left(  \{x, y\} =\{u, v\} \quad\text{or}\quad x,y \in A \quad\text{or}\quad x,y \in B  \right).
	\end{equation*}
	Hence,
	\begin{align*}
	2 \vert E(\Lambda(T)) \vert & = n(n-1) - 2 - a(a-1) - b(b-1)\\
	& \geq n(n-1) - 2(a-1) -a(a-1) -b(a-1) \\
	& \geq n(n-1) - n(a-1) = n(n-a),
	\end{align*}
	which yields
	\begin{equation*}
	\vert E(\Lambda(T)) \vert \geq \frac{1}{2}n(n-a),
	\end{equation*}
	as desired.
	
	Finally, suppose that $ T $ is of Type 2 and $ T \neq P_{5}, P_{6} $. So, $ a \geq 2 $ and $ D = a $. Therefore,
	\begin{equation*}
	xy \notin E(\Lambda(T)) \Leftrightarrow \left( \{x, y\} =\{u, v\} \quad\text{or}\quad x,y \in A \quad\text{or}\quad x,y \in B\right).
	\end{equation*}
	Hence,
	\begin{align*}
	2 \vert E(\Lambda(T)) \vert & \geq n(n-1) - 2 - a(a-1) - b(b-1)\\
	& \geq n(n-1) - (a+b+2)(a-1) \\
	& > n(n-1) - n(a-1) = n(n-a),
	\end{align*}
	which implies
	\begin{equation*}
	\vert E(\Lambda(T)) \vert \geq \frac{1}{2}n(n-a),
	\end{equation*}
	as desired.
\end{proof}
The next lemma provides a lower bound for the number of odd pairs in a graph $ G $ which have the same first component.
\begin{lemma}\label{l-2}
	Let $ G $ be a graph of order $ n $ and $ e=uv $ be an edge in $ E(\Lambda(G)) $. Then, there exist at least $ n - 3 $ odd pairs in $ G $ such that their first component is $ X = \lbrace u, v \rbrace $.
\end{lemma}
\begin{proof}
	Assume that $ (X, Y) $ is an odd pair of $ G $, where $ Y = \lbrace u^{\prime}, v^{\prime} \rbrace $. Then, $ \vert E(X, Y) \vert $ is an odd number and the parity of $ E(X, \lbrace u^{\prime} \rbrace) $ and $ E(X, \lbrace v^{\prime} \rbrace) $ are different. Hence, for every $ w \in V(G) \setminus (X \cup Y) $, exactly one of the two pairs
	\begin{equation*}
	(X, \lbrace u^{\prime}, w \rbrace) , \quad (X, \lbrace v^{\prime}, w \rbrace) 
	\end{equation*}
	is an odd pair. Therefore, $ G $ has at least $ n - 4 $ odd pairs different from $ (X, Y) $ and the proof is complete.
\end{proof}
\begin{remark}\label{r-aekn}
	We notify that in the procedure of the proof of Theorem 2 of \cite{aekn}, for a graph $ G $ of order $ n \geq 4 $, the inequality
	\begin{equation*}
	\se{G} \geq n - 4 + \sqrt{n^{2} -2n + 4 + 4\sqrt{\frac{3}{4}n^{2} +\nop{G} }}.
	\end{equation*}
	was obtained. With this inequality in hand and the above lemmas, we are ready to prove our main theorem:
\end{remark}
\begin{theorem}\label{nft}
	Let $ T $ be a tree of order $ n$. If $ D = D(T) $, then
	$$ \se{T} \geq 2n-6 + \sqrt{2(n-D)}. $$
\end{theorem}
\begin{proof}
	First, note that if $ T $ is a star graph or one of the paths $ P_{4} $, $ P_{5} $ or $ P_{6} $, then by \cite{aekn}, $ \se{T} \geq 2n-2 $ which is greater than the lower bound given in theorem.
	
	So, assume that $n\geq 4$ and $ T $ is not a star graph or a path of order less than 7. By Lemmas \ref{ntsize} and \ref{l-2}, we have
	\begin{align}\label{p6}
	\nop{T} \geq \vert E(\Lambda(T)) \vert (n-3) & \geq \frac{1}{2}n(n-3)(n-D) \nonumber \\
	& \geq \frac{1}{2}(n-2)^{2}(n-D). \qquad \quad (n \geq 4)
	\end{align}
	Now, inequality \ref{p6} and Remark \ref{r-aekn} imply that
	\begin{align*}
	\se{T} & \geq n - 4 + \sqrt{n^{2} - 2n +4+ 4\sqrt{\nop{T}}} \\
	& \geq n - 4 + \sqrt{(n-2)^{2} + 2n + 4\sqrt{\frac{1}{2}(n-2)^{2}(n-D)}} \\
	& \geq n - 4 + \sqrt{(n-2)^{2} + 2(n-D) + 2(n-2)\sqrt{2(n-D)}}\\
	& = n - 4 + (n-2) + \sqrt{2(n-D)}\\
	& = 2n-6 + \sqrt{2(n - D)},
	\end{align*}
	which completes the proof.
\end{proof}

We close the paper by a numerical discussion about the average value of $ D(T) $. Let $ T $ be a random tree of order $ n $, where by \textit{random} we mean a \textit{uniformly} chosen spanning tree of $ K_{n} $. Thanks to Gordon Royle answer\footnote{ https://mathoverflow.net/a/402550/125843}
% and Peter Taylor comments
% \footnote{https://mathoverflow.net/a/402890/125843}
and using the software SageMath \cite{sage},  we have the following table:
\vskip 2mm
%\begin{table}[h]
%\centering
%\begin{tabular}{| c | c || c | c |}
%\hline
%n & \textrm{Average value of $D(T)$} & n & \textrm{Average value of $D(T)$} 
%\\\hline
%10  &  2.575471698113207 & 120  &  4.35344882358862\\\hline
%20  &  2.967998882226798 & 130  &  4.42188243291670\\\hline
%30  &  3.244348960626259 & 140  &  4.48556779925447 \\\hline
%40  &  3.456978431955974 & 150  &  4.54513863025468\\\hline
%50  &  3.630127314849775 & 160  &  4.60110325186865 \\\hline
%60  &  3.776031348645077 & 170  &  4.65387595489334\\\hline
%70  &  3.901815507930055 & 180  &  4.70379959393469 \\\hline
%80  &  4.012168617452553 & 190  &  4.75116198280184 \\\hline
%90  &  4.110393093383418 & 210  &  4.83914784839474\\\hline
%100  &  4.198895438962389 & 220 &  4.88016483445798\\\hline
%110  &  4.279464181321791 & 230  &  4.91941918100074 \\\hline

%\end{tabular}
%\caption{The average value of $ D(T) $.}\label{t1}
%\end{table}
\begin{table}[h]
\centering
\begin{tabular}{| c | c || c | c |}
\hline
n & \textrm{Average value of $D(T)$} & n & \textrm{Average value of $D(T)$} 
\\\hline
6 & 1.5509259259259258 & 15 & 1.8099355316779862\\\hline
7 & 1.601832569762599 & 16 & 1.8342600346899551\\\hline
8 & 1.615203857421875 & 17 & 1.8576535568845751 \\\hline
9 & 1.6481679057505914 & 18 & 1.8801276901061494\\\hline
10 & 1.6749189 & 19 & 1.9017080817999203 \\\hline
11 & 1.7038043317645422 & 20 &  1.9224240041946314\\\hline
12 & 1.7314162738607985 & 21 & 1.9423085461667031 \\\hline
13 & 1.7585077425737015 & 22  & 1.9613962948137142 \\\hline
14 & 1.7846671875609421 & 23  & 1.9797225829436216
\\\hline
\end{tabular}
\caption{The average value of $ D(T) $.}\label{t1}
\end{table}
\begin{figure}[h]
	\centering
	\includegraphics[height=7.5cm]{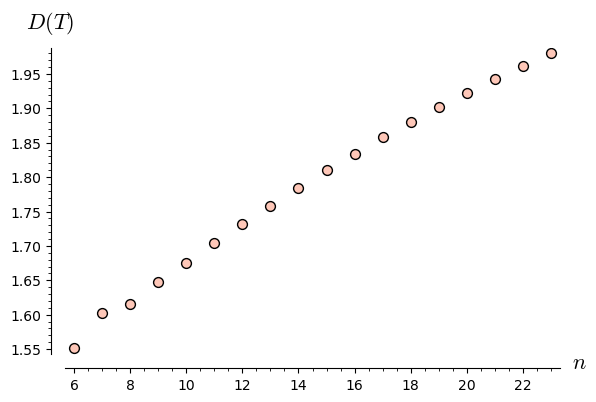}
	\caption{Graph of $D(T) $ in terms of $n$.}
\end{figure}

Table \ref{t1} shows clearly that the average value of $ D(T) $ is much less than $ n $. At the time of writing this paper, we do not know the value of $ \lim\limits_{n \rightarrow \infty} \frac{D(T)}{n} $. We therefore encourage motivated readers to calculate the limit as a future study.

\end{document}